\pdfoutput=1
\pdfoutput=1
\pdfoutput=1
\pdfoutput=1
\pdfoutput=1
\documentclass[conference]{IEEEtran}
\usepackage{cite}
\usepackage{amsmath}
\usepackage{amssymb}
\usepackage{amsthm}
\usepackage{arydshln}
\usepackage{mathtools}
\usepackage{mathrsfs}
\usepackage{algorithmic}
\usepackage{graphicx}
\usepackage{epsfig,graphics}
\usepackage{epstopdf}
\usepackage{subfigure}
\usepackage{enumitem}
\usepackage{pifont}

\usepackage{multirow,ragged2e}
\usepackage{bm}
\usepackage{fixltx2e}
\usepackage[english]{babel}

\usepackage{filecontents}
\theoremstyle{theorem}
\newtheorem{theorem}{Theorem}
\theoremstyle{lemma}
\newtheorem{lemma}{Lemma}
\theoremstyle{definition}
\newtheorem{definition}{Definition}
\theoremstyle{assumption}
\newtheorem{assumption}{Assumption}
\theoremstyle{problem}
\newtheorem{problem}{Problem}
\theoremstyle{example}

\theoremstyle{proposition}
\newtheorem{proposition}{Proposition}
\theoremstyle{corollary}
\newtheorem{corollary}{Corollary}

\newtheorem{remark}{Remark}

\newcommand{\sect}[1]{Section~\ref{#1}}

\newcommand{\eq}[1]{Equation~(\ref{#1})}
\newcommand{\eqs}[2]{Equations~(\ref{#1})-(\ref{#2})}
\newcommand{\ass}[1]{Assumption~\ref{#1}}

\newcommand{\prop}[1]{Proposition~\ref{#1}}
\newcommand{\lem}[1]{Lemma~\ref{#1}}
\newcommand{\rem}[1]{Remark~\ref{#1}}
\newcommand{\prob}[1]{Problem~\ref{#1}}
\newcommand{\theo}[1]{Theorem~\ref{#1}}
\newcommand{\defi}[1]{Definition~\ref{#1}}

\newcommand{\fbm}[1]{\mathbf{#1}}
\newcommand{\dbm}[1]{\dot{\fbm{#1}}}
\newcommand{\ddbm}[1]{\ddot{\fbm{#1}}}

\newcommand{\tbm}[1]{\fbm{#1}^\mathsf{T}}

\newcommand{\tfbm}[1]{\bm{#1}^\mathsf{T}}
\newcommand{\tilbm}[1]{\tilde{\fbm{#1}}}

\newcommand{\ttilbm}[1]{\tilde{\fbm{#1}}^\mathsf{T}}
\newcommand{\dtbm}[1]{\dot{\fbm{#1}}^\mathsf{T}}

\newcommand{\dfbm}[1]{\dot{\bm{#1}}}
\newcommand{\bbm}[1]{\overline{\fbm{#1}}}
\newcommand{\tbbm}[1]{\overline{\fbm{#1}}^\mathsf{T}}

\begin{document}

\title{Connectivity-Preserving Swarm Teleoperation With A Tree Network}

\author{Yuan~Yang, Daniela Constantinescu and Yang Shi\thanks{The authors are with the Department of Mechanical Engineering, University of Victoria, Victoria, BC V8W 2Y2 Canada (e-mail: yangyuan@uvic.ca; yshi@uvic.ca; danielac@uvic.ca).}}

\maketitle

\begin{abstract}
During swarm teleoperation, the human operator may threaten the distance-dependent inter-robot communications and, with them, the connectivity of the slave swarm. To prevent the harmful component of the human command from disconnecting the swarm network, this paper develops a constructive strategy to dynamically modulate the interconnections of, and the locally injected damping at, all slave robots. By Lyapunov-based set invariance analysis, the explicit law for updating that control gains has been rigorously proven to synchronize the slave swarm while preserving all interaction links in the tree network. By properly limiting the impact of the user command rather than rejecting it entirely, the proposed control law enables the human operator to guide the motion of the slave swarm to the extent to which it does not endanger the connectivity of the swarm network. 
\end{abstract}

\IEEEpeerreviewmaketitle

\section{Introduction}\label{sec:introduction}
Fully autonomous multi-robot systems~(MRS-s) have been extensively studied because they are more robust and flexible than single-robot systems~\cite{Schwager2011Proceedings, Kumar2011IJRR}. Compared with autonomous MRS-s, semi-autonomous teleoperated swarms are partially controlled by human operators and, therefore, better suited for dealing with complex problems in unpredictable environments~\cite{Paolo2012RAM}. A first swarm bilateral teleoperation controller has connected the master and slave sides through velocity-like variables in~\cite{Paolo2011RSS, Paolo2012TRO}, and has steared bearing formations in~\cite{Paolo2012IJRR}. A less conservative swarm teleoperator has decoupled the master and slave sides through virtual kinematic points in~\cite{Lee2013TMECH}. Maneuverability and perceptual sensitivity for three types of haptic cues has been studied in~\cite{Paolo2013TMECH}. Time-varying density functions have improved the agility of human-swarm adaptive interaction for optimal coverage control in~\cite{Egerstedt2015TRO}. A virtual rigid body abstraction bridging the master and slave sides has enabled flying an arbitrary number of aerial robots with collision avoidance in~\cite{Schwager2016ICRA}. Force exchanges between a human and a system of unmanned aerial vehicles have been studied and validated in~\cite{Heinrich2017IJRR}. 

Because distributed synchronization of MRS-s demands inter-robot information exchanges which are practically constrained by the inter-robot distances~\cite{Paolo2011RSS}, robot teams with limited communication range need coordination strategies with guaranteed connectivity maintenance. For semi-autonomous MRS-s, existing work employs passivity-based control for both global~\cite{Paolo2011RSS, Paolo2013ICRA} and local~\cite{Lee2013TMECH} connectivity preservation. A potential function of the estimate of the algebraic connectivity of the swarm teleoperation system yields gradient-based controls that preserve global connectivity in~\cite{Paolo2011RSS}. Its extension in~\cite{Paolo2013ICRA} regulates the degree of connectivity of the remote group of robots. Because the algebraic connectivity is obtained by the algorithm in~\cite{Freeman2010Automatica}, the effectiveness of the designed controls and thus the safety of the teleoperated swarm rely on the accuracy and convergence rate of the connectivity estimates~\cite{Dimos2014Automatica}. As discussed in~\cite{Sabattini2017TRO}, the error in the connectivity estimates foils the ability to mitigate perturbations of global connectivity control even for first-order MRS-s. A potential function of relative distances between virtual kinematic points yields controls that are proven to maintain only the local connectivity of the virtual system in~\cite{Lee2013TMECH}.  

The main contribution of the paper is a constructive dynamic coupling and damping injection law for connectivity-preserving swarm teleoperation with a tree network. Without loss of generality, the proposed design assumes that only one slave robot can receive the user command from the master side. To maintain all interaction links in the tree network of the slave swarm, a customized potential function using inter-robot distances is employed to quantify the energy stored in these links. Then, set invariance analysis demonstrates that the connectivity of the swarm network can be preserved by properly upper-bounding the energy stored in the tree topology. Using the structural controllability of a tree network, the potential energy stored in the slave swarm has been provably upper-bounded by properly-designed control inputs using local information. To find an explicit law for updating control gains, the model reduction technique transforms the system dynamics into a first-order representation with state-dependent mismatches. A critical step in the design is that the impact of the state-dependent mismatches on connectivity preservation can be suppressed by dynamically modulating both the interconnections and the local damping injection of all slave robots according to their distances. Therefore, the proposed dynamic control can restrain the user-injected energy into the tree network of the slave swarm, and can prevent the negative impact of the operator command from disconnecting the tree network.

\section{Problem Formulation}\label{sec: problem formulation}
Consider a swarm teleoperation system that comprises one master robot and $N$ slave robots. An operator can command the group of slave robots to a desired location by operating the master robot. In the slave swarm, there is one informed slave robot that can receive the user command from the master side regardless of their distance~\cite{Lee2013TMECH}. Assume that the master robot and the informed slave have been passively connected. This paper then focuses mainly on preserving the connectivity of the slave swarm under the user perturbation transmitted from the master side. 

Let the slave swarm be a network of $n$-degree-of-freedom~($n$-DOF) Euler-Lagrange (EL) systems: 
\begin{equation}\label{equ1}
\begin{aligned}
\fbm{M}_{1}(\fbm{x}_{1})\cdot\ddbm{x}_{1}+\fbm{C}_{1}(\fbm{x}_{1},\dbm{x}_{1})\cdot\dbm{x}_{1}=&\fbm{u}_{1}+\fbm{f}\textrm{,}\\
\fbm{M}_{s}(\fbm{x}_{s})\cdot\ddbm{x}_{s}+\fbm{C}_{s}(\fbm{x}_{s},\dbm{x}_{s})\cdot\dbm{x}_{s}=&\fbm{u}_{s}\textrm{.}
\end{aligned}
\end{equation}
In \eq{equ1}: the subscript $1$ indicates the informed slave that receives user command from the master side; subscripts $s=2,\cdots,N$ index the remaining $N-1$ uninformed slave robots; and $\fbm{f}$ is the time-varying user command from the master side. For every slave robot $i=1,\cdots,N$: its position, velocity and acceleration vectors are denoted by $\fbm{x}_{i}$, $\dbm{x}_{i}$ and $\ddbm{x}_{i}$, respectively; $\fbm{M}_{i}(\fbm{x}_{i})$ and $\fbm{C}_{i}(\fbm{x}_{i},\dbm{x}_{i})$ are the matrices of inertia and of Coriolis and centrifugal effects, respectively; and $\fbm{u}_{i}$ is the control force of the robot to be designed. The dynamics~\eqref{equ1} have the following properties:
\begin{enumerate}[label=P.\arabic*]
\item \label{P1}
The inertia matrices $\fbm{M}_{i}(\fbm{x}_{i})$ are symmetric, positive definite and uniformly bounded by: $\lambda_{i1}\fbm{I}\preceq\fbm{M}_{i}(\fbm{x}_{i})\preceq\lambda_{i2}\fbm{I}$ for any $\fbm{x}_{i}\in\mathbb{R}^{n}$, where $\lambda_{i1}>0$ and $\lambda_{i2}>0$;
\item \label{P2}
$\dbm{M}_{i}(\fbm{x}_{i})-2\fbm{C}_{i}(\fbm{x}_{i},\dbm{x}_{i})$ are skew-symmetric;
\item \label{P3}
$\fbm{C}_{i}(\fbm{x}_{i},\fbm{y}_{i})$ are linear in $\fbm{y}_{i}$, and there exist $c_{i}>0$ such that $\|\fbm{C}_{i}(\fbm{x}_{i},\fbm{y}_{i})\cdot\fbm{z}_{i}\|\leq c_{i}\cdot\|\fbm{y}_{i}\|\cdot\|\fbm{z}_{i}\|$, $\forall \fbm{x}_{i}, \fbm{y}_{i}, \fbm{z}_{i}\in\mathbb{R}^{n}$.
\end{enumerate}  

Let the information exchanges among all slave robots in the system be constrained by the same communication radius $r$. Then, slave robots $i$ and $j$ can exchange their information at time instant $t\geq 0$ only if their distance is strictly smaller than $r$, i.e., $\|\fbm{x}_{ij}\|=\|\fbm{x}_{i}(t)-\fbm{x}_{j}(t)\|<r$, $\forall i,j\in\{1,\cdots,N\}$. Say, slave $i$ and slave $j$ are adjacent to each other, and their bidirectional communication link $(i,j)$ exists if and only if there are information exchanges between them at time $t\geq 0$. Because all inter-slave interactions are limited by an identical communication radius $r$, this paper assumes and aims to maintain certain bidirectional interaction links $(i,j)$ of the slave swarm network by properly constraining $\|\fbm{x}_{ij}(t)\|$ for all time $t\geq 0$.  

Information exchanges among slave robots in the swarm can be represented by an undirected graph $\mathcal{G}(t)=\{\mathcal{V},\mathcal{E}(t)\}$. The vertex set $\mathcal{V}=\{1,\cdots,N\}$ collects all slave robots in the system. The edge set $\mathcal{E}(t)\subset\{(i,j)\in\mathcal{V}\times\mathcal{V}\}$ includes all communication links among slave robots. By definition, $(i,j)\in\mathcal{E}(t)$ if and only if robot $i$ and $j$ are exchanging information. For each slave robot $i\in\mathcal{V}$, its neighbourhood $\mathcal{N}_{i}(t)=\{j\in\mathcal{V}\ |\ (i,j) \in \mathcal{E}(t)\}$ collects all slave robots that are adjacent to it at time instant $t\geq 0$. In $\mathcal{G}(t)$, a path between two slaves $i$ and $j$ is a sequence of vertices $i,a,b,\cdots,j$ such that consecutive vertices are adjacent in which there is no repeated vertex and edge. Then, the graph $\mathcal{G}(t)$ is said to be connected if and only if there is a path between every two distinct vertices. Further, $\mathcal{G}(t)$ is a tree if every two vertices are connected by exactly one path.

Given a tree $\mathcal{G}(t)$ of order $N$, the associated weighted adjacency matrix $\fbm{A}=[a_{ij}]$ with $a_{ij}$ the $(i,j)$-th element of $\fbm{A}$ is defined by: $a_{ij}=a_{ji}>0$ if $(i,j)\in\mathcal{E}(t)$, and $a_{ij}=0$ otherwise. The weighted Laplacian matrix $\fbm{L}=[l_{ij}]$ with $l_{ij}$ the $(i,j)$-th element of $\fbm{L}$ is defined by $l_{ij}=\sum_{k\in\mathcal{N}_{i}(t)}a_{ik}$ if $j=i$, and $l_{i,j}=-a_{ij}$ otherwise. In particular, if $a_{ij}\in\{0,1\}$, $\fbm{L}$ becomes an unweighted Laplacian matrix $\bbm{L}$. Let an orientation of $\mathcal{G}(t)$ define an oriented graph $\mathcal{G}^{\ast}(t)$, and label each oriented edge $(i,j)$ as $e_{k}$, $k=1,\cdots,N-1$, with weight $w(e_{k})=a_{ij}=a_{ji}$. Let the associated incidence matrix be $\fbm{D}=[d_{hk}]$ in which the $(h,k)$-th element $d_{hk}$ is defined by: $d_{hk}=1$ if vertex $h$ is the head of edge $e_{k}$; $d_{hk}=-1$ if $h$ is the tail of $e_{k}$ and $d_{hk}=0$ otherwise. The edge Laplacian of the oriented graph is then defined as $\fbm{L}_{e}=\tbm{D}\fbm{D}$. The follows lemmas~\cite{Egerstedt2010Princeton} will be used to prove~\lem{lem1} in~\sect{sec: main results}:
\begin{enumerate}[label=L.\arabic*]
\item \label{L1}
The second smallest eigenvalue $\lambda_{L}$ of the unweighted Laplacian $\bbm{L}$ is positive, i.e., $\lambda_{L}>0$.
\item \label{L2}
The set of nonzero eigenvalues of $\fbm{L}_{e}$ is equal to the set of nonzero eigenvalues of the unweighted Laplacian $\bbm{L}$.
\item \label{L3}
The weighted Laplacian matrix admits the decomposition $\fbm{L}=\fbm{D}\fbm{W}\tbm{D}$, with $\fbm{W}$ an $(N-1)\times (N-1)$ diagonal matrix with $w(e_{k})$, $k=1,\cdots,N-1$, on the diagonal.
\end{enumerate} 

The following assumptions on the initial configuration of the system and on the user force are adopted in this paper.
\begin{assumption}\label{ass1}
The initial interaction network $\mathcal{G}(0)$ of the slave swarm is a tree and each pair of initially adjacent robots $(i,j)\in\mathcal{E}(0)$ is strictly within their communication distance, i.e., $\|\fbm{x}_{ij}(0)\|< r-\epsilon$ for some $\epsilon>0$.
\end{assumption} 
\begin{assumption}\label{ass2}
The user command from the master side is bounded by $\|\fbm{f}\|\leq\overline{f}$.
\end{assumption}

Because every connected network contains at least one spanning tree~\cite{Egerstedt2010Princeton}, this paper assumes directly the initial interaction network $\mathcal{G}(0)$ to be a tree in~\ass{ass1}. Considering the inertia of second-order dynamics, it further adopts the same distance condition $\|\fbm{x}_{ij}(0)\|<r-\epsilon$ on every pair of initially adjacent robots $(i,j)\in\mathcal{E}(0)$ as in connectivity preservation of fully autonomous second-order MRS-s~\cite{Su2010SCL}. Then, the connectivity-preserving swarm teleoperation problem addressed in this paper can be formulated as:
\begin{problem}\label{prob1}
Given the teleoperated swarm system~\eqref{equ1} satisfying Assumptions~\ref{ass1} and~\ref{ass2}, find distributed control laws to drive the system such that:
\begin{enumerate}
\item[1.]
The velocities of, and the position errors between, every two slave robots $i,j=1,\cdots,N$ are bounded in the presence of the user command, i.e., $\{\dbm{x}_{i}, \dbm{x}_{j}, \fbm{x}_{i}-\fbm{x}_{j}\}\in\mathcal{L}_{\infty}$ when $\fbm{f}\neq\fbm{0}$;
\item[2.]
All slave robots $i,j=1,\cdots,N$ are synchronized in the absence of the user command, i.e., $\{\dbm{x}_{i}, \dbm{x}_{j}, \fbm{x}_{i}-\fbm{x}_{j}\}\to\fbm{0}$ when $\fbm{f}=\fbm{0}$;
\item[3.]
All interaction links $(i,j)\in\mathcal{E}(0)$ of the initial slave network $\mathcal{G}(0)$ are maintained, i.e., $(i,j)\in\mathcal{E}(t)$ $\forall t\geq 0$ if $(i,j)\in\mathcal{E}(0)$, and, with them, the connectivity of the slave swarm $\mathcal{G}(t)$ is preserved. 
\end{enumerate}
\end{problem}

In~\prob{prob1}: the first two objectives are similar to those of conventional bilateral teleoperation systems~\cite{Lee2010TRO}; the last objective together with~\ass{ass1} illustrate that the proposed dynamic strategy can preserve the connectivity of any connected swarm teleoperation network by maintaining a spanning tree of the slave swarm. Future research will take advantage of switching spanning trees for connectivity-preserving swarm teleoperation.

The following definition of input-to-state stability will be used in~\sect{sec: main results} to quantify robust position synchronization of the teleoperated swarm.
\begin{definition}\label{iss}
~\cite{Sontag2008Springer} The perturbed nonlinear system
\begin{align*}
\dbm{x}(t)=f(\fbm{x}(t),\fbm{u}(t))
\end{align*}
is ISS with input $\fbm{u}(t)\in\mathbb{R}^{m}$ and state $\fbm{x}(t)\in\mathbb{R}^{n}$ if there exist functions $\alpha\in\mathcal{K}$ and $\beta\in\mathcal{KL}$ such that for any $t\geq 0$:
\begin{align*}
\|\fbm{x}(t)\|\leq\beta(\|\fbm{x}(0)\|,t)+\alpha\left(\sup\limits_{0\leq\tau\leq t}\|\fbm{u}(\tau)\|\right)\textrm{.}
\end{align*}
Further, if $\beta(\|\fbm{x}(0)\|,t)$ decreases exponentially with respect to $t$, the system is exponentially ISS. 
\end{definition}

\section{Main Results}\label{sec: main results}

By~\ass{ass1} and the last item in~\prob{prob1}, connectivity maintenance of the swarm teleoperation system is guaranteed by rendering invariant the edge set $\mathcal{E}(t)$ of the tree network $\mathcal{G}(0)$ for any $t\geq 0$. Because inter-robot communication links $(i,j)\in\mathcal{E}(t)$ are constrained by their distances, this paper employs the following functions for verifying the distance constraints:
\begin{equation}\label{equ2}
\psi(\|\fbm{x}_{ij}\|)=\frac{P\cdot\|\fbm{x}_{ij}\|^{2}}{r^{2}-\|\fbm{x}_{ij}\|^{2}+Q}\textrm{,}
\end{equation}
where $P$ and $Q$ are positive constants to be designed. For every $\|\fbm{x}_{ij}\|\in[0,r]$, the function $\psi(\|\fbm{x}_{ij}\|)$ is continuous, positive definite and strictly increasing with respect to $\|\fbm{x}_{ij}\|$~\cite{Su2010SCL}. Then, the potential energy stored in all links $(i,j)\in\mathcal{E}(0)$ can be described by
\begin{equation}\label{equ3}
V_{p}=\frac{1}{2}\cdot\sum^{N}_{i=1}\sum_{j\in\mathcal{N}_{i}(0)}\psi(\|\fbm{x}_{ij}\|)\textrm{.}
\end{equation}

The following proposition illustrates the feasibility of connectivity preservation using function~\eqref{equ2}:
\begin{proposition}\label{prop1}
Under~\ass{ass1} and given any $\Delta>0$, select $Q$ and $P$ as follows
\begin{equation}\label{equ4}
\begin{aligned}
&\left[r^{2}-(N-1)(r-\epsilon)^{2}\right]Q+\left[r^{2}-(r-\epsilon)^{2}\right]r^{2}>0\textrm{,}\\
&P>\frac{\left[r^{2}-(r-\epsilon)^{2}+Q\right]Q\cdot\Delta}{\left[r^{2}-(r-\epsilon)^{2}+Q\right]r^{2}-(N-1)Q\cdot(r-\epsilon)^{2}}\textrm{.}
\end{aligned}
\end{equation}
It guarantees then that
\begin{align*}
V_{p}(0)+\Delta<\frac{P\cdot r^{2}}{Q}=\psi_{\max}\textrm{.}
\end{align*}
\end{proposition}
\begin{proof}
\ass{ass1} and the property that~\eqref{equ2} imply that
\begin{align*}
V_{p}(0)<&\frac{1}{2}\cdot\sum^{N}_{i=1}\sum_{j\in\mathcal{N}_{i}(0)}\frac{P\cdot(r-\epsilon)^{2}}{r^{2}-(r-\epsilon)^{2}+Q}\\
=&\frac{P\cdot(N-1)\cdot(r-\epsilon)^{2}}{r^{2}-(r-\epsilon)^{2}+Q}\textrm{,}
\end{align*} 
where $N-1$ is the number of links in $\mathcal{G}(0)$. Let
\begin{align*}
\omega=\frac{P}{\left[r^{2}-(r-\epsilon)^{2}+Q\right]Q}>0\textrm{.}
\end{align*}
It follows then that
\begin{align*}
&\psi_{\max}-V_{p}(0)\geq \frac{P\cdot r^{2}}{Q}-\frac{P\cdot(N-1)\cdot(r-\epsilon)^{2}}{r^{2}-(r-\epsilon)^{2}+Q}\\
=&\omega\cdot\Big[r^{2}\left(r^{2}-(r-\epsilon)^{2}+Q\right)-(N-1)Q\cdot(r-\epsilon)^{2}\Big]\\
=&\omega\cdot\Big(\left[r^{2}-(N-1)(r-\epsilon)^{2}\right]Q+\left[r^{2}-(r-\epsilon)^{2}\right]r^{2}\Big)>0
\end{align*}
can be ensured by selecting $Q>0$ small enough such that
\begin{align*}
\begin{aligned}
&\left[r^{2}-(N-1)(r-\epsilon)^{2}\right]Q+\left[r^{2}-(r-\epsilon)^{2}\right]r^{2}>0\textrm{.}
\end{aligned}
\end{align*} 
After setting $Q$ as above, choosing $P>0$ with
\begin{align*}
P>\frac{\left[r^{2}-(r-\epsilon)^{2}+Q\right]Q\cdot\Delta}{\left[r^{2}-(r-\epsilon)^{2}+Q\right]r^{2}-(N-1)Q\cdot(r-\epsilon)^{2}}
\end{align*} 
can make $V_{p}(0)+\Delta<\psi_{\max}$.
\end{proof}

Based on~\prop{prop1}, the distance constraint on every link $(i,j)\in\mathcal{E}(0)$ can be examined by the proposition below:
\begin{proposition}\label{prop2}
Under~\ass{ass1} and given any $\Delta>0$, let $P$ and $Q$ be selected to satisfy~\eqref{equ4}. At any time $t\geq 0$, if 
\begin{equation}\label{equ5}
V_{p}(\tau)\leq V_{p}(0)+\Delta\textrm{,}\quad \forall\tau\in[0,t]\textrm{,}
\end{equation}
then $\|\fbm{x}_{ij}(t)\|<r$ for every $(i,j)\in\mathcal{E}(0)$.
\end{proposition}
\begin{proof}
By~\prop{prop1}, \eq{equ5} implies that
\begin{align*}
V_{p}(\tau)<\psi_{\max}\textrm{,}\quad \forall \tau\in[0,t]\textrm{.}
\end{align*}
\ass{ass1} and the property of~\eqref{equ2} further lead to
\begin{align*}
0\leq\psi(\|\fbm{x}_{ij}(0)\|)<\psi_{\max}\textrm{,}\quad \forall (i,j)\in\mathcal{E}(0)\textrm{.}
\end{align*}
Suppose that, at time instant $t$, link $(i,j)$ has among all links in $\mathcal{E}(0)$ the maximal length $\|\fbm{x}_{ij}(t)\|=r$. Then, $\psi(\|\fbm{x}_{ij}(t)\|)=\psi_{\max}$ and $0\leq\psi(\|\fbm{x}_{lm}(t)\|)\leq \psi_{\max}$ for any other $(l,m)\in\mathcal{E}(0)$ because $\|\fbm{x}_{lm}(t)\|\leq r$ and $\psi(\cdot)$ is continuous, positive definite and strictly increasing on $[0,r]$. Hence, 
\begin{align*}
V_{p}(t)=\psi(\|\fbm{x}_{ij}(t)\|)+\sum_{(l,m)\in\overline{\mathcal{E}}(0)}\psi(\|\fbm{x}_{lm}(t)\|)\geq \psi_{\max}\textrm{,}
\end{align*}
where $\overline{\mathcal{E}}(0)=\mathcal{E}(0)-\{(i,j)\}$, which contradicts~\eqref{equ5}. Therefore, $\|\fbm{x}_{ij}(t)\|<r$ for every $(i,j)\in\mathcal{E}(0)$
\end{proof}

Together, the above two propositions demonstrate the fundamental principle of proving connectivity preservation: set invariance~\cite{Blanchini1999Auto,Ames2017TAC}. In this paper, the edge set $\mathcal{E}(t)$ is rendered invariant by properly constraining the distance $\|\fbm{x}_{ij}\|$ between each pair of initially adjacent slave robots $(i,j)\in\mathcal{E}(0)$ for swarm teleoperation with a tree network. Here, the potential functions~\eqref{equ2} and~\eqref{equ3} characterize the inter-slave distances in Propositions~\ref{prop1} and~\ref{prop2} and can thus be employed to investigate the invariance of $\mathcal{E}(t)$. A similar approach has been proposed in~\cite{Su2010SCL} for proving connectivity maintenance of autonomous double-integrator multi-agent systems. In contrast, this paper focuses on preserving the tree network of a slave swarm that is driven by a time-varying and unpredictable user command. More specifically, the main contribution of this paper is constructively designing a dynamic coupling and damping injection controller by the controllability of a tree network to obtain~\eqref{equ5} under the perturbation of the user input transmitted from the master side.

To bound the potential energy $V_{p}$ using the local information of each robot $i$, define a surface $\fbm{s}_{i}$ for each slave $i$ by:
\begin{equation}\label{equ6}
\fbm{s}_{i}=\dbm{x}_{i}+\sigma\cdot\bm{\theta}_{i}\textrm{,}
\end{equation}
where $i=1,\cdots,N$, and $\sigma>0$ and 
\begin{equation}\label{equ7}
\bm{\theta}_{i}=\sum_{j\in\mathcal{N}_{i}(0)}\nabla_{i}\psi(\|\fbm{x}_{ij}\|)
\end{equation}
with the gradient of $\psi(\|\fbm{x}_{ij}\|)$ with respect to $\fbm{x}_{i}$ given by
\begin{equation}\label{equ8}
\nabla_{i}\psi(\|\fbm{x}_{ij}\|)=\frac{2P\cdot\left(r^{2}+Q\right)}{\left(r^{2}-\|\fbm{x}_{ij}\|^{2}+Q\right)^{2}}\cdot\left(\fbm{x}_{i}-\fbm{x}_{j}\right)\textrm{.}
\end{equation}
Then, the swarm teleoperation system dynamics~\eqref{equ1} can be transformed into
\begin{equation}\label{equ9}
\begin{aligned}
\fbm{M}_{1}(\fbm{x}_{1})\cdot\dbm{s}_{1}+\fbm{C}_{1}(\fbm{x}_{1},\dbm{x}_{1})\cdot\fbm{s}_{1}=&\sigma\cdot\fbm{\Delta}_{1}+\fbm{u}_{1}+\fbm{f}\textrm{,}\\
\fbm{M}_{s}(\fbm{x}_{s})\cdot\dbm{s}_{s}+\fbm{C}_{s}(\fbm{x}_{s},\dbm{x}_{s})\cdot\fbm{s}_{s}=&\sigma\cdot\fbm{\Delta}_{s}+\fbm{u}_{s}
\end{aligned}
\end{equation}
where $s=2,\cdots,N$ index uninformed slave robots, and $\bm{\Delta}_{i}$ are state-dependent mismatches
\begin{equation}\label{equ10}
\fbm{\Delta}_{i}=\fbm{M}_{i}(\fbm{x}_{i})\cdot\dfbm{\theta}_{i}+\fbm{C}_{i}(\fbm{x}_{i},\dbm{x}_{i})\cdot\bm{\theta}_{i}\textrm{.}
\end{equation}

The following lemma is key to proving connectivity maintenance in the remainder of the paper:
\begin{lemma}\label{lem1}
Given the tree network $\mathcal{G}(0)$ of the teleoperated swarm~\eqref{equ1}, the following inequality holds:
\begin{equation}\label{equ11}
\begin{aligned}
\sum^{N}_{i=1}\tfbm{\theta}_{i}\bm{\theta}_{i}\geq \frac{4\lambda_{L}P}{r^{2}+Q}\cdot V_{p}\textrm{.}
\end{aligned}
\end{equation}
\end{lemma}
\begin{proof}
Associated with the tree $\mathcal{G}(0)$, the weighted adjacency matrix $\fbm{A}=[a_{ij}]$ with $a_{ij}$ the $(i,j)$-th element is defined by
\begin{align*}
a_{ij}=\begin{cases}
\frac{2P\cdot(r^{2}+Q)}{\left(r^{2}-\|\fbm{x}_{ij}\|^{2}+Q\right)^{2}}\quad &\text{if } j\in\mathcal{N}_{i}(0)\textrm{,}\\
0\quad &\text{otherwise.}
\end{cases}
\end{align*}
Then, the corresponding weighted laplacian $\fbm{L}=[l_{ij}]$ with $l_{ij}$ the $(i,j)$-th element is given by
\begin{align*}
l_{ij}=\begin{cases}
-a_{ij}\quad &\text{if } j\neq i\\
\sum_{k\in\mathcal{N}_{i}(0)}a_{ik}\quad &\text{else if } j=i \textrm{.} 
\end{cases}
\end{align*} 

Let $\fbm{l}_{i}$ be the $i$-th row of $\fbm{L}$. It follows then that
\begin{align*}
\sum_{j\in\mathcal{N}_{i}(0)}\nabla_{i}\psi(\|\fbm{x}_{ij}\|)=\left(\fbm{l}_{i}\otimes\fbm{I}_{n}\right)\fbm{x}\textrm{,}
\end{align*}
where $\fbm{x}=[\tbm{x}_{1}\ \cdots\ \tbm{x}_{N}]^\mathsf{T}$. By the definition of $\bm{\theta}_{i}$ in~\eqref{equ7}, the left-hand side of~\eqref{equ11} becomes:
\begin{align*}
\sum^{N}_{i=1}\tfbm{\theta}_{i}\bm{\theta}_{i}&=\sum^{N}_{i=1}\tbm{x}\left(\fbm{l}_{i}\otimes\fbm{I}_{n}\right)^\mathsf{T}\left(\fbm{l}_{i}\otimes\fbm{I}_{n}\right)\fbm{x}\\
&=\tbm{x}\left[\left(\sum^{N}_{i=1}\tbm{l}_{i}\fbm{l}_{i}\right)\otimes\fbm{I}_{n}\right]\fbm{x}=\tbm{x}\left(\tbm{L}\fbm{L}\otimes\fbm{I}_{n}\right)\fbm{x}\textrm{.}
\end{align*}
By~Lemma~\ref{L3} and the definition of $\fbm{W}$ introduced in~\sect{sec: problem formulation}, it can further be re-organized as:
\begin{equation}\label{equ12}
\begin{aligned}
&\sum^{N}_{i=1}\tfbm{\theta}_{i}\bm{\theta}_{i}=\tbm{x}\left[\fbm{D}\fbm{W}\tbm{D}\fbm{D}\fbm{W}\tbm{D}\otimes\fbm{I}_{n}\right]\fbm{x}\\
=&\left[\left(\fbm{W}\tbm{D}\otimes\fbm{I}_{n}\right)\fbm{x}\right]^\mathsf{T}\left(\tbm{D}\fbm{D}\otimes\fbm{I}_{n}\right)\left[\left(\fbm{W}\tbm{D}\otimes\fbm{I}_{n}\right)\fbm{x}\right]\\
=&\tbbm{x}\left(\fbm{L}_{e}\otimes\fbm{I}_{n}\right)\bbm{x}\textrm{,}
\end{aligned}
\end{equation}
where $\bbm{x}=[\tbbm{x}_{1}\ \cdots\ \tbbm{x}_{N-1}]^\mathsf{T}=(\fbm{W}\tbm{D}\otimes\fbm{I}_{n})\fbm{x}$ with $\bbm{x}_{k}=\nabla_{i}\psi(\|\fbm{x}_{ij}\|)$ for $e_{k}=(i,j)$, $k=1,\cdots,N-1$. Here, $\bbm{x}$ stacks the weighted position mismatch between each pair of adjacent robots $(i,j)\in\mathcal{E}(0)$~\cite{Egerstedt2010Princeton}. 

By~Lemma~\ref{L1}, \ass{ass1} implies that the second smallest eigenvalue $\lambda_{L}$ of the unweighted laplacian $\overline{\fbm{L}}$ is positive, i.e., $\lambda_{L}>0$. Further, $\fbm{L}_{e}$ is an $(N-1)\times (N-1)$ matrix, because $\mathcal{G}(0)$ is a tree, and has smallest eigenvalue $\lambda_{L}$ by~Lemma~\ref{L2}. Therefore, the left-hand side of~\eqref{equ11} can be lower-bounded by:
\begin{align*}
&\sum^{N}_{i=1}\tfbm{\theta}_{i}\bm{\theta}_{i}\geq  \lambda_{L}\cdot\tbbm{x}\bbm{x}=\lambda_{L}\cdot\sum_{(i,j)\in\mathcal{E}(0)}\Big\|\nabla_{i}\psi(\|\fbm{x}_{ij}\|)\Big\|^{2}\\
=&\sum_{(i,j)\in\mathcal{E}(0)}\frac{4\lambda_{L}P\cdot(r^{2}+Q)^{2}}{\left(r^{2}-\|\fbm{x}_{ij}\|^{2}+Q\right)^{3}}\cdot\psi(\|\fbm{x}_{ij}\|)\geq \frac{4\lambda_{L}P}{r^{2}+Q}\cdot V_{p}\textrm{.}
\end{align*} 
\end{proof}

For a teleoperated swarm network, every slave robot~$i$ receives the information sent by its initial neighbours $j\in\mathcal{N}_{i}(0)$ if $\|\fbm{x}_{ij}(t)\|<r$ for all time $t\geq 0$. Then, $\psi(\|\fbm{x}_{ij}\|)$ in~\eqref{equ2} and $V_{p}$ in~\eqref{equ3} can be employed as follows to design controllers and to prove connectivity maintenance, respectively.

Assume that each link $(i,j)\in\mathcal{E}(0)$ has been maintained during the time interval $[0,t)$, i.e., $\|\fbm{x}_{ij}(\tau)\|<r$ for all $\tau\in[0,t)$ and every $(i,j)\in\mathcal{E}(0)$. It implies that the position $\fbm{x}_{j}(t)$ of $j\in\mathcal{N}_{i}(0)$ can be employed in the control of robot $i$ at time $t$ to prove that $\|\fbm{x}_{ij}(t)\|<r$ for all $(i,j)\in\mathcal{E}(0)$ by~\prop{prop2}. Then, connectivity maintenance can be proven by induction on time~\cite{Su2010SCL}. Thus, the following control is proposed to render positively invariant the edge set $\mathcal{E}(0)$:
\begin{equation}\label{equ13}
\fbm{u}_{i}=-K_{i}(t)\cdot\fbm{s}_{i}-D_{i}\dbm{x}_{i}-B_{i}\bm{\theta}_{i}\textrm{,}
\end{equation}  
where $i=1,\cdots,N$ index all slave robots, and $K_{i}(t)$, $D_{i}$ and $B_{i}$ are positive gains to be determined. 

\begin{remark}\label{rem1}
\normalfont
By the definition of $\fbm{s}_{i}$ in~\eqref{equ6}, the control law $\fbm{u}_{i}$ can be rewritten as
\begin{align*}
\fbm{u}_{i}=-\left[\sigma\cdot K_{i}(t)+B_{i}\right]\cdot\bm{\theta}_{i}-\left[K_{i}(t)+D_{i}\right]\cdot\dbm{x}_{i}\textrm{,}
\end{align*}
in which the first and the second terms are the coupling and the damping injection force, respectively. Note that here the argument $t$ is utilized to indicate that $K_{i}(t)$ is state-dependent and thus time-varying.   More specifically, $K_{i}(t)$ is dynamically updated according to the distances $\|\fbm{x}_{ij}\|$ between slave robot $i$ and its neighbours $j\in\mathcal{N}_{i}(0)$.
\end{remark}

Connectivity preservation is then investigated by the following Lyapunov candidate:
\begin{equation}\label{equ14}
\begin{aligned}
V=\frac{1}{2}\cdot\sum^{N}_{i=1}\frac{1}{B_{i}+\sigma D_{i}}\cdot\tbm{s}_{i}\fbm{M}_{i}(\fbm{x}_{i})\cdot\fbm{s}_{i}+V_{p}\textrm{,}
\end{aligned}
\end{equation}
in which $V_{p}$ has been defined in~\eqref{equ3}. Along the transformed system dynamics~\eqref{equ9} in closed-loop with the control~\eqref{equ13}, the derivative of $V$ is 
\begin{align*}
\dot{V}=&\frac{1}{2}\cdot\sum^{N}_{i=1}\frac{1}{B_{i}+\sigma D_{i}}\cdot\left[\tbm{s}_{i}\dbm{M}_{i}(\fbm{x}_{i})\cdot\fbm{s}_{i}+2\tbm{s}_{i}\fbm{M}_{i}(\fbm{x}_{i})\cdot\dbm{s}_{i}\right]\\
&+\frac{1}{2}\cdot\sum^{N}_{i=1}\sum_{j\in\mathcal{N}_{i}(0)}\left[\dtbm{x}_{i}\nabla_{i}\psi(\|\fbm{x}_{ij}\|)+\dtbm{x}_{j}\nabla_{j}\psi(\|\fbm{x}_{ij}\|)\right]\\
=&\sum^{N}_{i=1}\frac{1}{B_{i}+\sigma D_{i}}\cdot\left[\sigma\cdot\tbm{s}_{i}\fbm{\Delta}_{i}-K_{i}(t)\cdot\tbm{s}_{i}\fbm{s}_{i}\right]+\frac{\tbm{s}_{1}\fbm{f}}{B_{1}+\sigma D_{1}}\\
&-\sum^{N}_{i=1}\frac{\tbm{s}_{i}\left(D_{i}\dbm{x}_{i}+B_{i}\bm{\theta}_{i}\right)}{B_{i}+\sigma D_{i}}+\sum^{N}_{i=1}\sum_{j\in\mathcal{N}_{i}(0)}\dtbm{x}_{i}\nabla_{i}\psi(\|\fbm{x}_{ij}\|)\textrm{,}
\end{align*}
where Property~\ref{P1} and~\ref{P2} of~\eqref{equ1} and~\ass{ass1} have been applied. The definition of $\fbm{s}_{i}$ in~\eqref{equ6} implies that
\begin{align*}
&\tbm{s}_{i}(D_{i}\dbm{x}_{i}+B_{i}\bm{\theta}_{i})\\
=&D_{i}\dtbm{x}_{i}\dbm{x}_{i}+\sigma D_{i}\dtbm{x}_{i}\bm{\theta}_{i}+B_{i}\dtbm{x}_{i}\bm{\theta}_{i}+\sigma B_{i}\tfbm{\theta}_{i}\bm{\theta}_{i}\\
=&D_{i}\dtbm{x}_{i}\dbm{x}_{i}+\sigma B_{i}\tfbm{\theta}_{i}\bm{\theta}_{i}+(B_{i}+\sigma D_{i})\cdot\sum_{j\in\mathcal{N}_{i}(0)}\dtbm{x}_{i}\nabla_{i}\psi(\|\fbm{x}_{ij}\|)\textrm{.}
\end{align*}
Here it utilizes the fact that $\bm{\theta}_{i}$ in~\eqref{equ7} can be employed in the control $\fbm{u}_{i}$ at time instant $t$ based on the assumption that $\mathcal{E}(\tau)=\mathcal{E}(0)$ for all $\tau\in [0,t)$. Thus
\begin{equation}\label{equ15}
\begin{aligned}
\dot{V}=&\sum^{N}_{i=1}\frac{\sigma\cdot\tbm{s}_{i}\bm{\Delta}_{i}-\sigma B_{i}\tfbm{\theta}_{i}\bm{\theta}_{i}}{B_{i}+\sigma D_{i}}+\frac{\tbm{s}_{1}\fbm{f}}{B_{1}+\sigma D_{1}}\\
&-\sum^{N}_{i=1}\frac{1}{B_{i}+\sigma D_{i}}\cdot\big[K_{i}(t)\cdot\tbm{s}_{i}\fbm{s}_{i}+D_{i}\dtbm{x}_{i}\dbm{x}_{i}\big]\textrm{.}
\end{aligned}
\end{equation}

By the definition of $\bm{\theta}_{i}$ in~\eqref{equ7}, its derivative is 
\begin{equation}\label{equ16}
\begin{aligned}
\dfbm{\theta}_{i}=&\sum_{j\in\mathcal{N}_{i}(0)}\frac{8P\cdot(r^{2}+Q)\cdot\tbm{x}_{ij}\dbm{x}_{ij}\fbm{x}_{ij}}{\left(r^{2}-\|\fbm{x}_{ij}\|^{2}+Q\right)^{3}}\\
&+\sum_{j\in\mathcal{N}_{i}(0)}\frac{2P\cdot(r^{2}+Q)\cdot(\dbm{x}_{i}-\dbm{x}_{j})}{\left(r^{2}-\|\fbm{x}_{ij}\|^{2}+Q\right)^{2}}\textrm{.}
\end{aligned}
\end{equation}
Simple algebraic manipulations lead to
\begin{equation}\label{equ17}
\begin{aligned}
&\tbm{s}_{i}\fbm{M}_{i}(\fbm{x}_{i})\cdot\dfbm{\theta}_{i}\leq \sum_{j\in\mathcal{N}_{i}(0)}\left[2\left(\eta_{i}+\gamma_{i}\right)\left(\dtbm{x}_{i}\dbm{x}_{i}+\dtbm{x}_{j}\dbm{x}_{j}\right)\right]\\
&+\sum_{j\in\mathcal{N}_{i}(0)}\Bigg[\frac{16\lambda^{2}_{i2}P^{2}\cdot(r^{2}+Q)^{2}\cdot\|\fbm{x}_{ij}\|^{4}}{\eta_{i}\cdot\left(r^{2}-\|\fbm{x}_{ij}\|^{2}+Q\right)^{6}}\cdot\tbm{s}_{i}\fbm{s}_{i}\\
&\quad \quad \quad \quad \quad +\frac{\lambda^{2}_{i2}P^{2}\cdot\left(r^{2}+Q\right)^{2}}{\gamma_{i}\cdot\left(r^{2}-\|\fbm{x}_{ij}\|^{2}+Q\right)^{4}}\cdot\tbm{s}_{i}\fbm{s}_{i}\Bigg]
\end{aligned}
\end{equation}
with $\eta_{i}>0$ and $\gamma_{i}>0$, and that
\begin{equation}\label{equ18}
\begin{aligned}
&\tbm{s}_{i}\fbm{C}_{i}(\fbm{x}_{i},\dbm{x}_{i})\cdot\bm{\theta}_{i}\\
\leq &\sum_{j\in\mathcal{N}_{i}(0)}\Bigg[\frac{c^{2}_{i}P^{2}\cdot\left(r^{2}+Q\right)^{2}\cdot\|\fbm{x}_{ij}\|^{2}}{2\zeta_{i}\cdot\left(r^{2}-\|\fbm{x}_{ij}\|^{2}+Q\right)^{4}}\cdot\tbm{s}_{i}\fbm{s}_{i}+2\zeta_{i}\dtbm{x}_{i}\dbm{x}_{i}\Bigg]
\end{aligned}
\end{equation}
with $\zeta_{i}>0$, by Property~\ref{P1} and~\ref{P3} of~\eqref{equ1}, respectively. Hence, the impact of the mismatch $\fbm{\Delta}_{i}$ given in~\eqref{equ10} can be upper-bounded by
\begin{equation}\label{equ19}
\begin{aligned}
\tbm{s}_{i}\fbm{\Delta}_{i}\leq&\sum_{j\in\mathcal{N}_{i}(0)}\Big[\Lambda_{ij}(t)\cdot\tbm{s}_{i}\fbm{s}_{i}+2(\eta_{i}+\gamma_{i})\cdot\dtbm{x}_{j}\dbm{x}_{j}\\
&\quad \quad \quad \quad +2(\eta_{i}+\gamma_{i}+\zeta_{i})\cdot\dtbm{x}_{i}\dbm{x}_{i}\Big]\textrm{,}
\end{aligned}
\end{equation}
where 
\begin{align*}
&\Lambda_{ij}(t)=\frac{16\lambda^{2}_{i2}P^{2}\cdot(r^{2}+Q)^{2}\cdot\|\fbm{x}_{ij}\|^{4}}{\eta_{i}\cdot\left(r^{2}-\|\fbm{x}_{ij}\|^{2}+Q\right)^{6}}\\
&+\frac{\lambda^{2}_{i2}P^{2}\cdot\left(r^{2}+Q\right)^{2}}{\gamma_{i}\cdot\left(r^{2}-\|\fbm{x}_{ij}\|^{2}+Q\right)^{4}}+\frac{c^{2}_{i}P^{2}\cdot\left(r^{2}+Q\right)^{2}\cdot\|\fbm{x}_{ij}\|^{2}}{2\zeta_{i}\cdot\left(r^{2}-\|\fbm{x}_{ij}\|^{2}+Q\right)^{4}}\textrm{.}
\end{align*}
The user-injected energy can be measured by
\begin{equation}\label{equ20}
\frac{\tbm{s}_{1}\fbm{f}}{B_{1}+\sigma D_{1}}\leq \frac{1}{4\Gamma}\cdot\|\fbm{f}\|^{2}+\frac{\Gamma}{\left(B_{1}+\sigma D_{1}\right)^{2}}\cdot\tbm{s}_{1}\fbm{s}_{1}\textrm{,}
\end{equation} 
where $\Gamma>0$. Then, $\dot{V}$ can be upper-bounded below by substitution of~\eqref{equ19} and~\eqref{equ20} in~\eqref{equ15}:
\begin{equation}\label{equ21}
\begin{aligned}
\dot{V}\leq&-\sum^{N}_{i=1}\frac{\overline{K}_{i}(t)\cdot\tbm{s}_{i}\fbm{s}_{i}+\overline{D}_{i}\dtbm{x}_{i}\dbm{x}_{i}+\sigma B_{i}\tfbm{\theta}_{i}\bm{\theta}_{i}}{B_{i}+\sigma D_{i}}+\frac{\|\fbm{f}\|^{2}}{4\Gamma}\textrm{,}
\end{aligned}
\end{equation}  
where 
\begin{equation}\label{equ22}
\begin{aligned}
\overline{K}_{i}(t)=& K_{i}(t)-\sigma\cdot\sum_{j\in\mathcal{N}_{i}(0)}\Lambda_{ij}(t)-\frac{\Gamma_{i}}{\left(B_{1}+\sigma D_{1}\right)^{2}}\textrm{,}\\
\overline{D}_{i}=&D_{i}-2\sigma\cdot\sum_{j\in\mathcal{N}_{i}(0)}(\eta_{i}+\gamma_{i}+\zeta_{i}+\eta_{j}+\gamma_{j})
\end{aligned}
\end{equation}
with $\Gamma_{i}=\Gamma$ if $i=1$ and $\Gamma_{i}=0$ otherwise.

With~\lem{lem1}, invariance of the edge set $\mathcal{E}(0)$, and thus connectivity maintenance of the swarm teleoperation system~\eqref{equ1} can be validated by the following theorem:
\begin{theorem}\label{theorem1}
Under Assumptions~\ref{ass1} and~\ref{ass2}, connectivity of the teleoperated swarm~\eqref{equ1} can be maintained by rendering invariant the edge set $\mathcal{E}(0)$, if the proposed control law~\eqref{equ13} is designed for every slave $i=1,\cdots,N$ as follows:
\begin{enumerate}
\item[1. ]
pick $\rho$, $\sigma$, $\eta_{i}$, $\gamma_{i}$, $\zeta_{i}$, $\Gamma$ and $B_{i}$ heuristically;
\item[2. ]
set $D_{i}$ to make $\overline{D}_{i}\geq 0$ in~\eqref{equ22};
\item[3. ]
Select $Q$ by condition~\eqref{equ4};
\item[4. ]
choose $P$ sufficiently large such that
\begin{equation}\label{equ23}
P\geq \frac{\rho\cdot(r^{2}+Q)}{4\lambda_{L}}\cdot\max\limits_{i=1,\cdots,N}\left(\frac{B_{i}+\sigma D_{i}}{\sigma B_{i}}\right)\textrm{,}
\end{equation}
and that~\eq{equ4} is guaranteed with
\begin{align*}
\Delta=\frac{1}{2}\cdot\sum^{N}_{i=1}\frac{\lambda_{i2}}{B_{i}+\sigma D_{i}}\cdot\|\fbm{s}_{i}(0)\|^{2}+\frac{\overline{f}^{2}}{4\rho\Gamma}\textrm{;}
\end{align*}
\item[5. ]
update $K_{i}(t)$ according to~\eqref{equ22} to ensure that
\begin{equation}\label{equ24}
\overline{K}_{i}(t)\geq \frac{1}{2}\cdot\rho\cdot\lambda_{i2}\textrm{.}
\end{equation}
\end{enumerate}
\end{theorem}
\begin{proof}
Substitution of~\eqref{equ11} in~\eqref{equ21} leads to
\begin{equation}\label{equ25}
\begin{aligned}
\dot{V}\leq &-\sum^{N}_{i=1}\frac{\overline{K}_{i}(t)}{B_{i}+\sigma D_{i}}\cdot\tbm{s}_{i}\fbm{s}_{i}-\sum^{N}_{i=1}\frac{\overline{D}_{i}}{B_{i}+\sigma D_{i}}\cdot\dtbm{x}_{i}\dbm{x}_{i}\\
&-\min\limits_{i=1,\cdots,N}\left(\frac{\sigma B_{i}}{B_{i}+\sigma D_{i}}\right)\cdot\frac{4\lambda_{L}P}{r^{2}+Q}\cdot V_{p}+\frac{\|\fbm{f}\|^{2}}{4\Gamma}\\
\leq &-\frac{1}{2}\cdot\sum^{N}_{i=1}\frac{\rho\cdot\lambda_{i2}}{B_{i}+\sigma D_{i}}\cdot\tbm{s}_{i}\fbm{s}_{i}-\rho\cdot V_{p}+\frac{\|\fbm{f}\|^{2}}{4\Gamma}\\
\leq &-\rho\cdot V+\rho\cdot\chi\left(\|\fbm{f}\|\right)\textrm{,}
\end{aligned}
\end{equation}
where $\overline{D}_{i}\geq 0$, \eqs{equ23}{equ24} have been applied, and
\begin{align*}
\chi(\|\fbm{f}\|)=\frac{\|\fbm{f}\|^{2}}{4\rho\Gamma}\textrm{.}
\end{align*}
Time integration of $\dot{V}$ from $0$ to $t\geq 0$ gives that:
\begin{equation}\label{equ26}
\begin{aligned}
V(t)\leq & e^{-\rho t}\cdot V(0)+\rho\cdot\int^{t}_{0}e^{-\rho(t-\tau)}\cdot\chi(\|\fbm{f}(\tau)\|)d\tau\\
\leq & e^{-\rho t}\cdot V(0)+\rho\cdot\sup\limits_{0\leq\tau\leq t}\chi(\|\fbm{f}(\tau)\|)\cdot\int^{t}_{0}e^{-\rho(t-\tau)}d\tau\\
\leq &e^{-\rho t}\cdot V(0)+\sup\limits_{0\leq\tau\leq t}\chi(\|\fbm{f}(\tau)\|)\textrm{.}
\end{aligned}
\end{equation}
Then, at any time instant $t\geq 0$, \ass{ass2} as well as \eq{equ26} lead to
\begin{align*}
V_{p}(t)\leq V(t)\leq e^{-\rho t}\cdot V_{p}(0)+\Delta\textrm{.}
\end{align*}
With $Q$ and $P$ selected by~\eqref{equ4}, it implies by~\prop{prop2} that $\|\fbm{x}_{ij}(t)\|<r$ for every $(i,j)\in\mathcal{E}(0)$, and with them, the connectivity of the teleoperated slave swarm is preserved.
\end{proof}

The first two objectives of~\prob{prob1} can be proven by showing input-to-state stability of the slave swarm. Define the state of the slave swarm subsystem to be $\bm{\phi}=[\dtbm{x}, \ \ttilbm{x}]^\mathsf{T}$ in which 
\begin{align*}
\dbm{x}=&[\dtbm{x}_{1},\ \cdots,\ \dtbm{x}_{N}]^\mathsf{T}\in\mathbb{R}^{Nn}\textrm{,}\\
\tilbm{x}=&\left(\tbm{D}\otimes\fbm{I}_{n}\right)\fbm{x}\in\mathbb{R}^{(N-1)n}\textrm{,}
\end{align*}
stack the velocity $\dbm{x}_{i}$ of each slave $i$, and the position error $\fbm{x}_{ij}=\fbm{x}_{i}-\fbm{x}_{j}$ between every pair of adjacent slave robots $(i,j)\in\mathcal{E}(0)$, respectively. In~\cite{Lee2010TRO}, the same definition of the state of interactive robotic systems have been proposed for investigating system stability.

\begin{corollary}
Under Assumptions~\ref{ass1} and~\ref{ass2}, if parameters $\rho$, $\sigma$, $\eta_{i}$, $\gamma_{i}$, $\zeta_{i}$, $\Gamma$, $P$ and $Q$, and gains $K_{i}(t)$, $B_{i}$ and $D_{i}$ are selected as in~\theo{theorem1}, then the slave swarm~\eqref{equ1} in closed-loop with the proposed control~\eqref{equ13} is exponentially ISS with input $\fbm{f}$ and state $\bm{\phi}$.
\end{corollary}   
\begin{proof}
Let all parameters and control gains be choosen as in~\theo{theorem1}. Under Assumptions~\ref{ass1} and~\ref{ass2}, the proposed Lyapunov candidate $V$ can then be upper-bounded by~\eqref{equ26}. 

Let $\overline{\lambda}_{L}>0$ be the maximum eigenvalue of the edge laplacian $\fbm{L}_{e}$. It follows from~\eqref{equ14} then that
\begin{equation}\label{equ27}
\begin{aligned}
&\sum^{N}_{i=1}\tfbm{\theta}_{i}\bm{\theta}_{i}\leq \overline{\lambda}_{L}\cdot\tbbm{x}\bbm{x}=\frac{\overline{\lambda}_{L}}{2}\cdot\sum^{N}_{i=1}\sum_{j\in\mathcal{N}_{i}(0)}\|\nabla_{i}\psi(\|\fbm{x}_{ij}\|)\|^{2}\\
=&\sum^{N}_{i=1}\sum_{j\in\mathcal{N}_{i}(0)}\frac{2\overline{\lambda}_{L}P\cdot(r^{2}+Q)^{2}}{\left(r^{2}-\|\fbm{x}_{ij}\|^{2}+Q\right)^{3}}\cdot\psi(\|\fbm{x}_{ij}\|)\leq\frac{4\overline{\lambda}_{L}P}{r^{2}+Q}\cdot V_{p}\textrm{.}
\end{aligned}
\end{equation} 
The definition of $\fbm{s}_{i}$ in~\eqref{equ6} together with~\eqref{equ27} imply that
\begin{equation}\label{equ28}
\begin{aligned}
\sum^{N}_{i=1}\dtbm{x}_{i}\dbm{x}_{i}\leq &2\cdot\sum^{N}_{i=1}\left(\tbm{s}_{i}\fbm{s}_{i}+\sigma^{2}\cdot\tfbm{\theta}_{i}\bm{\theta}_{i}\right)\\
\leq &\sum^{N}_{i=1}2\tbm{s}_{i}\fbm{s}_{i}+\frac{8\sigma^{2}\overline{\lambda}_{L}P}{r^{2}+Q}\cdot V_{p}\textrm{.}
\end{aligned}
\end{equation}
Using~\eqref{equ28} and~\ref{P1} of~\eqref{equ1}, $V$ can be lower-bounded by
\begin{equation}\label{equ29}
\begin{aligned}
V\geq &\frac{1}{2}\cdot\min\limits_{i=1,\cdots,N}\left(\frac{\lambda_{i1}}{B_{i}+\sigma D_{i}}\right)\cdot\sum^{N}_{i=1}\tbm{s}_{i}\fbm{s}_{i}+V_{p}\\
\geq &\min\left[\frac{1}{4}\cdot\min\limits_{i=1,\cdots,N}\left(\frac{\lambda_{i1}}{B_{i}+\sigma D_{i}}\right), \frac{r^{2}+Q}{8\sigma^{2}\overline{\lambda}_{L}P}\right]\cdot\sum^{N}_{i=1}\dtbm{x}_{i}\dbm{x}_{i}\textrm{,}
\end{aligned}
\end{equation}
and also by
\begin{equation}\label{equ30}
\begin{aligned}
V\geq &V_{p}=\frac{1}{2}\cdot\sum^{N}_{i=1}\sum_{j\in\mathcal{N}_{i}(0)}\frac{P\cdot\|\fbm{x}_{ij}\|^{2}}{r^{2}-\|\fbm{x}_{ij}\|^{2}+Q}\\
\geq &\frac{P}{r^{2}+Q}\cdot\frac{1}{2}\cdot\sum^{N}_{i=1}\sum_{j\in\mathcal{N}_{i}(0)}\|\fbm{x}_{ij}\|^{2}=\frac{P}{r^{2}+Q}\cdot\ttilbm{x}\tilbm{x}\textrm{.}
\end{aligned}
\end{equation}
Therefore, the state $\bm{\phi}$ can be quantified by $V$ as follows:
\begin{equation}\label{equ31}
\|\bm{\phi}\|^{2}=\sum^{N}_{i=1}\dtbm{x}_{i}\dbm{x}_{i}+\ttilbm{x}\tilbm{x}\leq \kappa_{1}\cdot V\textrm{,}
\end{equation}
where \eqs{equ29}{equ30} have been applied, and
\begin{align*}
\kappa_{1}=\max\left[\max\limits_{i=1,\cdots,N}\left(\frac{4(B_{i}+\sigma D_{i})}{\lambda_{i1}}\right),\frac{8\sigma^{2}\overline{\lambda}_{L}P}{r^{2}+Q}\right]+\frac{r^{2}+Q}{P}\textrm{.}
\end{align*}

Similarly, the definition of $\fbm{s}_{i}$ in~\eqref{equ6} implies also that
\begin{equation}\label{equ32}
\begin{aligned}
\sum^{N}_{i=1}\tbm{s}_{i}\fbm{s}_{i}\leq 2\cdot\sum^{N}_{i=1}\left(\dtbm{x}_{i}\dbm{x}_{i}+\sigma^{2}\cdot\tfbm{\theta}_{i}\bm{\theta}_{i}\right)\textrm{.}
\end{aligned}
\end{equation}
Given that $\mathcal{E}(t)$ is invariant, $\|\fbm{x}_{ij}\|<r$ for any $(i,j)\in\mathcal{E}(0)$ and any $t\geq 0$, it follows then that
\begin{equation}\label{equ33}
\begin{aligned}
V_{p}=&\frac{1}{2}\cdot\sum^{N}_{i=1}\sum_{j\in\mathcal{N}_{i}(0)}\frac{P\cdot\|\fbm{x}_{ij}\|^{2}}{r^{2}-\|\fbm{x}_{ij}\|^{2}+Q}\\
\leq &\frac{P}{2Q}\cdot\sum^{N}_{i=1}\sum_{j\in\mathcal{N}_{i}(0)}\|\fbm{x}_{ij}\|^{2}=\frac{P}{Q}\cdot\ttilbm{x}\tilbm{x}\textrm{.}
\end{aligned}
\end{equation}
Equations~\eqref{equ27}, \eqref{equ32} and~\eqref{equ33} together upper-bound $V$ by
\begin{equation}\label{equ34}
\begin{aligned}
V\leq &\frac{1}{2}\sum^{N}_{i=1}\lambda_{i2}\cdot\tbm{s}_{i}\fbm{s}_{i}+V_{p}\\
\leq &\lambda_{2}\cdot\sum^{N}_{i=1}\dtbm{x}_{i}\dbm{x}_{i}+\sigma^{2}\lambda_{2}\cdot\sum^{N}_{i=1}\tfbm{\theta}_{i}\bm{\theta}_{i}+V_{p}\\
\leq &\lambda_{2}\cdot\sum^{N}_{i=1}\dtbm{x}_{i}\dbm{x}_{i}+\left(\frac{4\sigma^{2}\lambda_{2}\overline{\lambda}_{L}P}{r^{2}+Q}+1\right)\cdot V_{p}\\
\leq &\lambda_{2}\cdot\sum^{N}_{i=1}\dtbm{x}_{i}\dbm{x}_{i}+\left[\frac{4\sigma^{2}\lambda_{2}\overline{\lambda}_{L}P^{2}}{(r^{2}+Q)Q}+\frac{P}{Q}\right]\cdot\ttilbm{x}\tilbm{x}\leq \kappa_{2}\cdot\|\bm{\phi}\|^{2}
\end{aligned}
\end{equation}
with $\lambda_{2}=\max\limits_{i=1,\cdots,N}\left(\lambda_{i2}\right)$ and
\begin{align*}
\kappa_{2}=\max\left[\lambda_{2},\frac{4\sigma^{2}\lambda_{2}\overline{\lambda}_{L}P^{2}}{(r^{2}+Q)Q}+\frac{P}{Q}\right]\textrm{.}
\end{align*}
Define $\alpha(\cdot)\in\mathcal{K}_{\infty}$ and $\beta(\cdot,\cdot)\in\mathcal{KL}$ by
\begin{align*}
\alpha\left(\sup\limits_{0\leq\tau\leq t} \|\fbm{f}(\tau)\|\right)=&\sqrt{\frac{\kappa_{1}}{4\rho\Gamma}}\cdot\sup\limits_{0\leq\tau\leq t}\|\fbm{f}(\tau)\|\textrm{,}\\
\beta\left(\|\bm{\phi}(0)\|,t\right)=&\sqrt{\frac{\kappa_{1}\kappa_{2}}{e^{\rho t}}}\cdot\|\bm{\phi}(0)\|\textrm{.}
\end{align*}
Then, substitutions of~\eqref{equ31} and~\eqref{equ34} in~\eqref{equ26} lead to:
\begin{equation}\label{equ35}
\begin{aligned}
&\|\bm{\phi}(t)\|\leq \sqrt{\kappa_{1}\cdot V(t)}\\
\leq&\sqrt{\kappa_{1}\cdot e^{-\rho t}\cdot V(0)+\kappa_{1}\cdot\sup\limits_{0\leq\tau\leq t}\chi(\|\fbm{f}(\tau)\|)}\\
\leq &\sqrt{\kappa_{1}\cdot e^{-\rho t}\cdot V(0)}+\sqrt{\kappa_{1}\cdot\sup\limits_{0\leq\tau\leq t}\chi(\|\fbm{f}(\tau)\|)}\\
\leq &\sqrt{\kappa_{1}\kappa_{2}\cdot e^{-\rho t}\cdot\|\bm{\phi}(0)\|^{2}}+\alpha\left(\sup\limits_{0\leq\tau\leq t} \|\fbm{f}(\tau)\|\right)\\
=&\beta\left(\|\bm{\phi}(0)\|,t\right)+\alpha\left(\sup\limits_{0\leq\tau\leq t} \|\fbm{f}(\tau)\|\right)\textrm{.}
\end{aligned}
\end{equation}
Therefore, the slave swarm~\eqref{equ1} under~\eqref{equ13} is exponentially ISS with input $\fbm{f}$ and state $\bm{\phi}$ by~\defi{iss}.
\end{proof}

\lem{lem1} and \theo{theorem1} are the most significant contributions of this paper. The inequality~\eqref{equ11} holding for the tree network $\mathcal{G}(0)$ contributes to further bouding $\dot{V}$ in~\eqref{equ21} by~\eqref{equ25}. Then, time integration~\eqref{equ26} together with~\prop{prop2} indicate that the distance between every pair of slave robots $(i,j)\in\mathcal{E}(0)$ can be constrained to be strictly smaller than the communication radius $r$. Compared to fully autonomous MRS-s and leader-follower systems, the main challenge of preserving the connectivity of a teleoperated swarm is caused by the unpredictable user perturbation $\fbm{f}$ transmitted from the master side. Different from external disturbances, like wind forces, the user perturbation $\fbm{f}$ commands the motion of the slave swarm by operating a master device. Hence, it should not be entirely rejected for the physical human-robot interaction. However, $\fbm{f}$ may endanger the connectivity of the slave swarm, for example by moving the informed slave such that some slave robots cannot keep up with it. The above analysis proves the interesting fact that, when properly designed, the distributed control~\eqref{equ13} eliminates the threat posed by the user command to connectivity. Condition~\eqref{equ24} on $K_{i}(t)$ exposes the uniqueness of the proposed design, especially in the P+d form of the control as in~\rem{rem1}. Namely, the design strengthens the couplings between slave robots and simultaneously increases the local damping injection in the swarm network based on their relative distances. To the authors' best knowledge, the control~\eqref{equ13} is the first strategy to maintain the connectivity of a teleoperated swarm with a state-dependent updating law of the coupling and damping gains.

\section{Conclusions}

This paper has presented a dynamic coupling and damping injection law for connectivity-preserving swarm teleoperation with a tree network. Using a customized potential function, this paper has firstly illustrated the principle of proving connectivity maintenance by set invariance. After reducing the order of the system dynamics by sliding surfaces, the mismatches induced by system dynamics transformations has been quantified by inter-robot distances. Then, the dynamic coupling and damping injection strategy has been designed to suppress the impact of the mismatches on connectivity maintenance. Rigorous energy analysis further forms the main contributions of this paper: the conclusion that all links of the tree network, and thus the connectivity of the teleoperated slave swarm, can be preserved by dynamic regulation of the inter-robot couplings and of the local damping injections. Future research will consider connectivity-preserving swarm teleoperation with limited actuation and heterogeneous communication radius.

\bibliography{IEEEabrv,Cooperation}
\end{document}